\documentclass[11pt]{amsart}
\usepackage[T1]{fontenc}
\usepackage{lmodern}
\usepackage{amssymb}
\usepackage[margin=1.2in]{geometry}
\usepackage{booktabs}
\usepackage{microtype}
\usepackage[hidelinks]{hyperref}

\DeclareMathOperator{\fix}{fix}
\DeclareMathOperator{\symop}{sym}

\DeclareMathOperator{\prof}{prof}
\newcommand{\U}{\mathfrak{U}}
\newcommand{\Vf}{\mathcal{V}_{F}}
\newcommand{\Vfp}{\mathcal{V}_{F^{+}}}
\newcommand{\Vp}{\mathcal{V}_{p}}
\newcommand{\Vpp}{\mathcal{V}_{p^{+}}}
\newcommand{\pw}{\mathcal{P}}
\newcommand{\restr}{\upharpoonright}
\newcommand{\seq}[1]{\langle #1\rangle}
\newcommand{\AC}{\mathsf{AC}}
\newcommand{\ACWO}{\AC_{\mathrm{WO}}}
\newcommand{\DC}{\mathsf{DC}}
\newcommand{\MC}{\mathsf{MC}}
\newcommand{\KW}{\mathsf{KW}}
\newcommand{\BPI}{\mathsf{BPI}}
\newcommand{\OP}{\mathsf{OP}}

\theoremstyle{plain}
\newtheorem{theorem}{Theorem}[section]
\newtheorem{lemma}[theorem]{Lemma}
\newtheorem{proposition}[theorem]{Proposition}
\newtheorem{corollary}[theorem]{Corollary}
\newtheorem{fact}[theorem]{Fact}
\theoremstyle{definition}
\newtheorem{definition}[theorem]{Definition}
\newtheorem{question}[theorem]{Question}
\theoremstyle{remark}
\newtheorem{remark}[theorem]{Remark}

\begin{document}

\title[Choice principles in countable partition models]{Choice principles in countable partition models}
\author{Brinda Venkataramani}

\date{October 21st, 2026}


\begin{abstract}
We study the countable partition model over a set of atoms of arbitrary uncountable cardinality $\kappa$. Its supports are partitions of the atoms into countably many blocks, and a permutation fixes a support when it preserves each block setwise. This model was proposed by Bruce and studied by Banerjee. We prove that dependent choice holds for every $\kappa$, although the filter of subgroups is not closed under countable intersections. Choice for well-orderable families holds exactly when $\kappa<\aleph_{\omega_1}$; when $\kappa\ge\aleph_{\omega_1}$, an $\omega_1$-indexed family witnesses its failure. The Hartogs and Lindenbaum numbers of the set of atoms are both $\omega_1$. The proofs use permutations matching blocks of equal cardinality and establish shift completeness of the filter. Together with Banerjee's failure of $\DC_{\omega_1}$, these results determine which dependent choice principles hold in the model.
\end{abstract}

\maketitle

\section{Introduction}\label{sec:intro}

Bruce \cite{Bruce} introduced a permutation model $\Vp$ whose supports are finite partitions, that is, partitions into finitely many blocks, of a countably infinite set of atoms. A support is fixed when each of its blocks is preserved setwise. Every ambient subset of the set of atoms belongs to this model, although the atoms cannot be linearly ordered. Bruce also proposed replacing the atoms by an uncountable set and allowing countable partitions as supports \cite[Section~5]{Bruce}. Banerjee \cite{Banerjee} studied this countable partition model, denoted by $\Vpp$, and proved several failures of choice, including the failure of choice for pairs of atoms and of $\DC_{\omega_1}$ \cite[Theorem~4.8]{Banerjee}.

We study the positive choice principles of $\Vpp$ as the ambient cardinality $\kappa$ of its set of atoms varies. Write $\ACWO$ for choice for well-orderable families of nonempty sets. We prove that $\Vpp\models\DC$ for every uncountable $\kappa$, and that
\[
  \Vpp\models\ACWO\ \Longleftrightarrow\ \kappa<\aleph_{\omega_1}.
\]
For $\kappa\ge\aleph_{\omega_1}$, we exhibit an $\omega_1$-indexed family without a choice function. Since countable choice holds for every $\kappa$, this is the least possible size of a well-orderable family witnessing failure of choice. Thus these models satisfy dependent choice together with failure of choice for well-orderable families. Below $\aleph_{\omega_1}$, the positive result also answers Banerjee's question about well-orderability of unions of well-orderable families of well-orderable sets \cite[Question~7.2]{Banerjee}.

The proofs compare refinements of a partition by counting their blocks of each cardinality. Two refinements with the same counts can be mapped to one another by a permutation fixing every block of the original partition. When $\kappa<\aleph_{\omega_1}$, only countably many infinite cardinals are available, and one countable partition suffices for all these comparisons. For an arbitrary $\kappa$, the same construction works for any given countable sequence of refinements. These two applications give $\ACWO$ below $\aleph_{\omega_1}$ and $\AC_\omega$ for every $\kappa$.

Dependent choice requires compatible permutations along a sequence of refining partitions. We prove that the filter of supports is shift complete and apply the theorem of Karagila and Schilhan \cite{KSDC}. The filter is not $\sigma$-complete, so the resulting model is not closed under countable sequences from the ambient universe. We also compute the Hartogs and Lindenbaum numbers of the atoms and record the standard failures of choice needed for the comparison with other permutation models. The last section collects the results and the remaining questions.

\section{Preliminaries}\label{sec:prelim}

We work in a universe $\U$ of $\mathsf{ZFA}+\AC$ whose set of atoms is $A$, with $|A|=\kappa$ uncountable. Cardinalities used to construct partitions and compare their blocks are computed in $\U$. Assertions about choice principles and well-orderability in $\Vpp$ are internal to that model. A set is \emph{pure} if no atom occurs in its transitive closure; the class of pure sets is called the \emph{kernel}.

A \emph{partition} is a family of pairwise disjoint nonempty sets covering the set being partitioned; its members are called \emph{blocks}. A partition is countable when it has countably many blocks, with finite partitions included. A partition $Q$ \emph{refines} $P$ if every block of $Q$ is contained in a block of $P$. Their common refinement is
\[
 P\wedge Q=\{p\cap q:p\in P,\ q\in Q,\ p\cap q\ne\emptyset\}.
\]
It is countable whenever $P$ and $Q$ are countable. If $Q$ refines $P$ and $p\in P$, we write $Q\restr p=\{q\in Q:q\subseteq p\}$.

We recall the standard permutation-model facts that we will use. A group of permutations of $A$ acts on $\U$ by $\in$-recursion. For an object $x$, its stabilizer is $\symop(x)=\{\pi:\pi x=x\}$. A filter of subgroups is \emph{normal} if it is closed under conjugation. An object is symmetric if its stabilizer belongs to the filter, and hereditarily symmetric if it and all objects in its transitive closure are symmetric.

\begin{fact}[{\cite[Lemma~2.5]{Bruce}; \cite[Chapter~4]{JechAC}}]\label{fact:basic}
Let $G$ be a group of permutations of $A$ and let $\mathcal F$ be a normal filter of subgroups containing the stabilizer of every atom. Let $\mathcal V$ be the class of hereditarily $\mathcal F$-symmetric objects. Then:
\begin{enumerate}
\item $\mathcal V$ is a transitive model of $\mathsf{ZFA}$ containing the kernel and $A$;
\item if $x\subseteq\mathcal V$, then $x\in\mathcal V$ if and only if $\symop(x)\in\mathcal F$;
\item a set $X\in\mathcal V$ is well-orderable in $\mathcal V$ if and only if some member of $\mathcal F$ fixes $X$ pointwise.
\end{enumerate}
\end{fact}

For our model, let $G=\operatorname{Sym}(A)$. For a partition $P$ of $A$, put
\[
 \fix(P)=\{\pi\in G:\pi[p]=p\text{ for every }p\in P\}.
\]
Let $\mathcal F$ consist of the subgroups of $G$ containing $\fix(P)$ for some countable partition $P$ of $A$.

\begin{lemma}[{\cite[Lemma~4.6]{Banerjee}}]\label{lem:filter}
The family $\{\fix(P):P\text{ a countable partition of }A\}$ is a basis for a normal filter containing the stabilizer of every atom.
\end{lemma}

\begin{proof}
We have $\fix(P\wedge Q)\subseteq\fix(P)\cap\fix(Q)$ and $\pi\fix(P)\pi^{-1}=\fix(\pi P)$. For each atom $a$, the partition $\{\{a\},A\setminus\{a\}\}$ has a stabilizer fixing $a$.
\end{proof}

\begin{definition}\label{def:model}
The model $\Vpp$ is the class of hereditarily $\mathcal F$-symmetric objects. A countable partition $P$ \emph{supports} $x$ if $\fix(P)\subseteq\symop(x)$.
\end{definition}

The parameter $\kappa$ is fixed throughout and is suppressed in the notation $\Vpp$. The kernel is the same in $\U$ and $\Vpp$; every pure set is well-orderable in $\Vpp$, and the two universes have the same ordinals and the same $\omega_1$. Normality gives $\pi\Vpp=\Vpp$ and $\symop(\pi x)=\pi\symop(x)\pi^{-1}$. We will use these facts without further comment.

For an infinite well-ordered cardinal $\lambda$, $\AC_\lambda$ means that every $\lambda$-indexed family of nonempty sets has a choice function. We write $\ACWO$ for the corresponding statement for every well-orderable index set. In the notation of \cite{HR}, this is $\mathsf C(\mathrm{WO},\infty)$.

The principle $\DC$ asserts that whenever $B\ne\emptyset$ and $R\subseteq B\times B$ has an $R$-successor for every member of $B$, there is $f:\omega\to B$ with $f(n)\mathrel R f(n+1)$ for every $n$. More generally, for an infinite well-ordered cardinal $\lambda$, $\DC_\lambda$ asserts that if $B\ne\emptyset$ and $R\subseteq{}^{<\lambda}B\times B$ satisfies
\[
 \forall s\in{}^{<\lambda}B\ \exists y\in B\ (s\mathrel R y),
\]
then there is $f:\lambda\to B$ with $(f\restr\xi)\mathrel R f(\xi)$ for every $\xi<\lambda$. Here ${}^{<\lambda}B=\bigcup_{\xi<\lambda}{}^\xi B$. We use only cardinal indices for these principles; $\DC_\omega$ is equivalent to $\DC$. See L\'evy \cite{Levy64} and \cite[Chapter~8]{JechAC}. We write $\DC_{<\lambda}$ for their conjunction over infinite cardinals below $\lambda$.

\section{Subsets of the set of atoms}\label{sec:subsets}

\begin{lemma}\label{lem:support}
Let $P$ be a countable partition of $A$ and $X\subseteq A$. Then $\fix(P)\subseteq\symop(X)$ if and only if $X$ is a union of blocks of $P$.
\end{lemma}

\begin{proof}
The forward implication follows by exchanging an atom of $X\cap p$ with an atom of $p\setminus X$ if both are nonempty, for a block $p\in P$. The reverse implication follows because every member of $\fix(P)$ preserves each block of $P$.
\end{proof}

\begin{corollary}\label{cor:allsubsets}
Every ambient subset of $A$ belongs to $\Vpp$; equivalently, $\pw(A)^{\Vpp}=\pw(A)^{\U}$. This power set also belongs to $\Vpp$.
\end{corollary}

\begin{proof}
The nonempty members of $\{X,A\setminus X\}$ form a support for $X$. The whole power set is fixed by $G$, so Fact~\ref{fact:basic}(2) applies to it as well.
\end{proof}

\begin{lemma}\label{lem:orbits}
Let $P$ be a countable partition of $A$.
\begin{enumerate}
\item The orbits of $\fix(P)$ on $A$ are the blocks of $P$.
\item The set of atoms fixed pointwise by $\fix(P)$ is the union of the singleton blocks of $P$, and is countable.
\item Every function on $A$ supported by $P$ and taking only pure values is constant on each block of $P$.
\end{enumerate}
\end{lemma}

\begin{proof}
Any two atoms in the same block can be exchanged by a transposition in $\fix(P)$, giving (1) and (2). For (3), if $a,b$ lie in the same block, let $\pi\in\fix(P)$ exchange them. Since $\pi f=f$ and $f(a)$ is pure,
\[
 f(a)=\pi(f(a))=(\pi f)(\pi a)=f(b).
\]
\end{proof}

\begin{proposition}\label{prop:proper}
Let $\Vfp$ be the permutation model over $A$ whose supports are countable subsets of $A$, with pointwise stabilizers. Then $\Vfp\subsetneq\Vpp$.
\end{proposition}

For $\kappa=\aleph_1$, the inclusion is due to Banerjee \cite[Theorem~4.8(1)]{Banerjee}.

\begin{proof}
For a countable $E\subseteq A$, its pointwise stabilizer equals the stabilizer of the countable partition
\[
 \{\{a\}:a\in E\}\cup\{A\setminus E\}.
\]
This gives $\Vfp\subseteq\Vpp$. Every subset of $A$ in $\Vfp$ is countable or co-countable, by the usual support argument for Fraenkel models; see \cite[Chapter~4]{JechAC}. Indeed, if $E$ is a countable support for $X\subseteq A$, then Lemma~\ref{lem:support}, applied to the displayed partition, gives $X\subseteq E$ or $A\setminus X\subseteq E$. An ambient subset of $A$ with both it and its complement uncountable belongs to $\Vpp$ and not to $\Vfp$.
\end{proof}

For $\kappa=\aleph_1$, the model $\Vfp$ is the countable-support Fraenkel model $\mathcal N12(\aleph_1)$ of \cite{HR}.

\section{Failures of choice}\label{sec:failures}

\begin{proposition}[{\cite[Theorem~4.8(3)]{Banerjee}}]\label{prop:pairs}
In $\Vpp$, there is no choice function on $[A]^2$.
\end{proposition}

\begin{proof}
Suppose that $c$ is such a choice function, supported by $P$. Some block of $P$ contains distinct atoms $a,b$. The transposition exchanging $a$ and $b$ belongs to $\fix(P)$ and fixes $\{a,b\}$ setwise, but moves $c(\{a,b\})$, a contradiction.
\end{proof}

The following consequences use the standard implications from the Boolean prime ideal theorem to the ordering principle, and from a linear order to a choice function on pairs; see \cite{JechAC}.

\begin{corollary}\label{cor:failures}
In $\Vpp$, the set $A$ admits no linear order, the ordering principle $\OP$ and the Boolean prime ideal theorem $\BPI$ fail, and neither $A$ nor $\pw(A)$ is well-orderable.
\end{corollary}

The multiple choice principle $\MC$ asserts that every family of nonempty sets admits a function selecting a nonempty finite subset of each member. The Kinna--Wagner selection principle $\KW$ asks for a nonempty proper subset of each member of a family whose members have at least two elements \cite{KW}. We also use the hierarchy $\KW_\alpha$, which asserts that every set injects into $\pw^\alpha(\eta)$ for some ordinal $\eta$ \cite{KSKW}. Here $\pw^0(\eta)=\eta$, successor stages take power sets, and limit stages take unions of the preceding stages.

\begin{proposition}\label{prop:mckw}
In $\Vpp$, multiple choice and the Kinna--Wagner selection principle fail. There is no injection of $A$ into any pure set, and $\KW_\alpha$ fails for every ordinal $\alpha$.
\end{proposition}

\begin{proof}
Suppose $c$ selects a nonempty finite subset of each nonempty subset of $A$, and let $P$ support $c$. Choose an infinite block $p\in P$. Every permutation of $p$, extended by the identity on $A\setminus p$, must preserve $c(p)$ setwise. The only subsets of $p$ invariant under all these permutations are $\emptyset$ and $p$, which contradicts the choice of $c(p)$.

A Kinna--Wagner selector on $[A]^2$ would give a choice function on pairs. Finally, every pure set is well-orderable in $\Vpp$, whereas $A$ is not. This rules out an injection of $A$ into any pure set, including every $\pw^\alpha(\eta)$.
\end{proof}

\section{Hartogs and Lindenbaum numbers}\label{sec:invariants}

\begin{definition}
For a set $X$, its \emph{Hartogs number} $\aleph(X)$ is the least ordinal that does not inject into $X$. Its \emph{Lindenbaum number} $\aleph^*(X)$ is the least nonzero ordinal onto which $X$ does not surject.
\end{definition}

\begin{theorem}\label{thm:invariants}
For every uncountable $\kappa$, the model $\Vpp$ satisfies
\[
 \aleph(A)=\aleph^*(A)=\omega_1.
\]
\end{theorem}

The upper bound on $\aleph(A)$ also occurs in Banerjee's proof of \cite[Theorem~4.8(5)]{Banerjee}.

\begin{proof}
By Fact~\ref{fact:basic}(3) and Lemma~\ref{lem:orbits}(2), every internally well-orderable subset of $A$ is contained in the union of the singleton blocks of some support, which is countable. Hence $\omega_1$ does not inject into $A$. Choose distinct atoms $\seq{a_n:n<\omega}$ in $\U$. The partition separating all $a_n$ as singletons supports this enumeration, so $\omega$ injects into $A$ in $\Vpp$. Thus $\aleph(A)=\omega_1$.

By Lemma~\ref{lem:orbits}(3), every ordinal-valued function on $A$ in $\Vpp$ has countable range. Thus $A$ does not surject onto $\omega_1$. The map sending $a_n$ to $n$ and every other atom to $0$ belongs to $\Vpp$ and maps $A$ onto $\omega$. Composing with surjections from $\omega$ onto nonzero countable ordinals gives $\aleph^*(A)=\omega_1$.
\end{proof}

In particular, $A$ is Dedekind-infinite in $\Vpp$.

\section{Partitions and conjugation}\label{sec:conjugation}

All the constructions in this section take place in $\U$.

\begin{definition}\label{def:profile}
Let $P,Q$ be partitions of $A$ with $Q$ refining $P$, and let $p\in P$. The \emph{profile} of $Q$ over $p$ records the number of blocks of each cardinality. As a multiplicity function, it is
\[
 \prof_p(Q)=
 \bigl\{\langle\mu,|\{q\in Q\restr p:|q|=\mu\}|\rangle:
           \mu\le |p|\text{ is a cardinal}\bigr\}.
\]
\end{definition}

\begin{lemma}\label{lem:profile}
Let $Q,Q'$ be countable refinements of a countable partition $P$. There is a permutation $\pi\in\fix(P)$ with $\pi Q=Q'$ if and only if $\prof_p(Q)=\prof_p(Q')$ for every $p\in P$.
\end{lemma}

\begin{proof}
Such a permutation preserves the cardinalities and multiplicities of the blocks inside each $p$. Conversely, if the profiles agree, match the blocks inside each $p$ by cardinality, choose bijections between the matched blocks, and take their union.
\end{proof}

The next lemma gives one construction for both applications needed in Section~\ref{sec:choice}. The countable set $D$ specifies the infinite cardinalities that may occur in the refinements under consideration.

\begin{lemma}\label{lem:refinement}
Let $P$ be a countable partition of $A$, and let $D$ be a countable set of infinite cardinals containing $|p|$ for every infinite block $p\in P$. There is a countable refinement $R$ of $P$ such that, whenever $Q$ is a countable refinement of $P$ and every infinite block of $Q$ has cardinality in $D$, some $\pi\in\fix(P)$ satisfies that $R$ refines $\pi Q$.
\end{lemma}

\begin{proof}
We construct $R$ separately on each block $p\in P$. For countable $p$, let $R\restr p$ consist of singletons. Suppose $\lambda=|p|$ is uncountable, and put
\[
 E_p=
 \begin{cases}
  D\cap\lambda,&\text{if }\sup(D\cap\lambda)=\lambda,\\
  D\cap(\lambda+1),&\text{otherwise}.
 \end{cases}
\]
Partition $p$ into countably infinitely many singletons and countably infinitely many pieces of cardinality $\mu$ for each $\mu\in E_p$. These cardinalities sum to $\lambda$: in the first case $E_p$ is cofinal in $\lambda$, and in the second case $\lambda\in E_p$. Taking these partitions over all $p$ gives a countable partition $R$ of $A$.

Fix a refinement $Q$ as in the statement. By Lemma~\ref{lem:profile}, it suffices to group the pieces of $R\restr p$ into cells having the same cardinalities as the blocks of $Q\restr p$, for every $p$. The countable blocks of $P$ require no argument, since their singletons can be grouped into the blocks of $Q$ themselves. Fix an uncountable block $p$, of cardinality $\lambda$. Since $Q\restr p$ is countable,
\[
 \sup\{|q|:q\in Q\restr p\}=\lambda.
\]
If $D\cap\lambda$ is not cofinal in $\lambda$, then $Q\restr p$ must have a block of size $\lambda$. Assign to each finite block of $Q\restr p$ the required number of singleton pieces, and to each infinite block one piece of its own cardinality. Make these assignments disjoint. Put all unassigned pieces into one of the cells of cardinality $\lambda$. The resulting partition has the required profile.

Suppose instead that $D\cap\lambda$ is cofinal in $\lambda$. Choose strictly increasing $\seq{\mu_j:j<\omega}$ from $D\cap\lambda$ cofinal in $\lambda$. For each $j$, divide the pieces of $R\restr p$ of cardinality $\mu_j$ into two countably infinite families, one for assignments of exact cardinality and one enumerated as $\{B_{j,t}:t<\omega\}$. Finite blocks again receive singleton pieces. A block of infinite cardinality $\nu<\lambda$ receives one piece of size $\nu$, using the first family when $\nu=\mu_j$. The $t$-th block of cardinality $\lambda$, if any, receives
\[
 C_t=\bigcup_{j<\omega}B_{j,t}.
\]
These cells are disjoint, have cardinality $\lambda$, and are disjoint from all assignments of smaller cardinality.

For each $\mu\in E_p$, add all remaining pieces of cardinality $\mu$ to one designated infinite cell of cardinality at least $\mu$. Such a cell exists because the target cardinalities have supremum $\lambda$. Add any remaining singletons to one infinite cell. A cell of cardinality $\nu$ receives leftovers from at most countably many classes, each contributing at most $\nu$ atoms, so its cardinality remains $\nu$. Finite cells receive no leftovers. We have partitioned $p$ into the required cells.

The resulting coarsening of $R$ has the same profile as $Q$ over every block of $P$, completing the proof.
\end{proof}

\begin{corollary}\label{cor:refinements}
Let $P$ be a countable partition of $A$.
\begin{enumerate}
\item If $\kappa<\aleph_{\omega_1}$, there is a countable refinement $R$ of $P$ such that for every countable refinement $Q$ of $P$, some $\pi\in\fix(P)$ satisfies that $R$ refines $\pi Q$.
\item For any uncountable $\kappa$ and any sequence $\seq{Q_n:n<\omega}$ of countable refinements of $P$, there are a countable refinement $R$ of $P$ and permutations $\pi_n\in\fix(P)$ such that $R$ refines $\pi_nQ_n$ for every $n$.
\end{enumerate}
\end{corollary}

\begin{proof}
For (1), take $D$ to be all infinite cardinals at most $\kappa$; this is countable because $\kappa<\aleph_{\omega_1}$. For (2), take $D$ to contain the infinite cardinalities of all blocks of $P$ and of all the $Q_n$. Apply Lemma~\ref{lem:refinement}.
\end{proof}

\section{Choice for countable and well-orderable families}\label{sec:choice}

\begin{theorem}\label{thm:ctble}
For every uncountable $\kappa$, the model $\Vpp$ satisfies $\AC_\omega$.
\end{theorem}

\begin{proof}
Let $\seq{X_n:n<\omega}\in\Vpp$ be a sequence of nonempty sets, supported by $P$. Then $\fix(P)$ fixes each $X_n$ setwise. In $\U$, choose $x_n\in X_n$ and a countable refinement $Q_n$ of $P$ supporting $x_n$. Apply Corollary~\ref{cor:refinements}(2), obtaining $R$ and $\pi_n\in\fix(P)$ with $R$ refining $\pi_nQ_n$.

Put $y_n=\pi_nx_n$. We have $y_n\in X_n$, and $R$ supports every $y_n$. Thus $R$ supports the function $f=\{\langle n,y_n\rangle:n<\omega\}$. Its elements are hereditarily symmetric, so $f\in\Vpp$ by Fact~\ref{fact:basic}(2).
\end{proof}

We record the following standard consequences of countable choice for comparison with the models in Section~\ref{sec:comparison}; see \cite{JechAC,HR}. Part~(2) settles the countable case of Banerjee's union question \cite[Question~7.2]{Banerjee} for every uncountable $\kappa$.

\begin{corollary}\label{cor:consequences}
For every uncountable $\kappa$, the model $\Vpp$ satisfies:
\begin{enumerate}
\item every countable union of countable sets is countable;
\item every countable union of well-orderable sets is well-orderable;
\item every Dedekind-finite set is finite, and hence there is no infinite amorphous set;
\item $\omega_1$ is regular.
\end{enumerate}
\end{corollary}

\begin{theorem}\label{thm:wochoice}
If $\kappa<\aleph_{\omega_1}$ is uncountable, then $\Vpp\models\ACWO$.
\end{theorem}

\begin{proof}
Let $W\in\Vpp$ be a well-orderable family of nonempty sets. By Fact~\ref{fact:basic}(3), there is a countable partition $P$ such that $\fix(P)$ fixes $W$ pointwise. Choose $R$ as in Corollary~\ref{cor:refinements}(1).

For $X\in W$, choose in $\U$ an element $x\in X$ and a countable refinement $Q$ of $P$ supporting $x$. There is $\pi\in\fix(P)$ such that $R$ refines $\pi Q$. Then $\pi x\in X$, and $R$ supports $\pi x$. Every member of $W$ therefore contains an element supported by the same partition $R$.

Using choice in $\U$, select such an element $c(X)$ for each $X\in W$. Since $R$ refines $P$, its stabilizer fixes every $X\in W$ and every $c(X)$. Thus $R$ supports $c$, and Fact~\ref{fact:basic}(2) gives $c\in\Vpp$.
\end{proof}

\begin{corollary}\label{cor:union}
If $\kappa<\aleph_{\omega_1}$ is uncountable, then in $\Vpp$ every union of a well-orderable family of well-orderable sets is well-orderable.
\end{corollary}

This gives a positive answer to Banerjee's question \cite[Question~7.2]{Banerjee} for $\kappa<\aleph_{\omega_1}$. We return to the remaining cases in Question~\ref{q:union}.

\section{Failure of choice for \texorpdfstring{$\omega_1$}{omega1}-indexed families}\label{sec:failurewo}

The next calculation takes place in $\U$.

\begin{lemma}\label{lem:sizes}
For every countable partition $P$, the set
\[
 \left\{\left|\bigcup\mathcal S\right|:\mathcal S\subseteq P\right\}
\]
is countable.
\end{lemma}

\begin{proof}
Let $B=\{|p|:p\in P\}$, a countable set of cardinals. A finite subfamily has union of cardinality in $\omega\cup B$. For an infinite subfamily $\mathcal S$,
\[
 \left|\bigcup\mathcal S\right|=\max\bigl(\aleph_0,\sup\{|p|:p\in\mathcal S\}\bigr).
\]
Every supremum of a subset of $B$ is the supremum of an initial segment of $B$. Since $B$, ordered by cardinality, has countable order type, it has only countably many initial segments. Only countably many cardinalities can therefore occur.
\end{proof}

\begin{theorem}\label{thm:failure}
If $\kappa\ge\aleph_{\omega_1}$, then $\Vpp\models\neg\AC_{\omega_1}$. Consequently $\ACWO$ fails in $\Vpp$.
\end{theorem}

\begin{proof}
For $\gamma<\omega_1$, let
\[
 S_\gamma=\{X\subseteq A:|X|^{\U}=\aleph_\gamma\}.
\]
These sets are nonempty and pairwise disjoint. Every $X\in S_\gamma$ belongs to $\Vpp$ by Corollary~\ref{cor:allsubsets}. Each $S_\gamma$ is fixed by $G$, since permutations preserve ambient cardinalities. Hence the sequence $\seq{S_\gamma:\gamma<\omega_1}$ belongs to $\Vpp$.

Suppose that $c$ chooses $X_\gamma\in S_\gamma$ for every $\gamma$, and that $P$ supports $c$. Since the indices are pure, $\fix(P)$ fixes each $X_\gamma$ setwise. By Lemma~\ref{lem:support}, each $X_\gamma$ is a union of blocks of $P$. Lemma~\ref{lem:sizes} says that their ambient cardinalities range over a countable set, contrary to $|X_\gamma|^{\U}=\aleph_\gamma$ for all $\gamma<\omega_1$.
\end{proof}

\begin{corollary}\label{cor:dichotomy}
For every uncountable $\kappa$,
\[
 \Vpp\models\ACWO\quad\Longleftrightarrow\quad\kappa<\aleph_{\omega_1}.
\]
If $\kappa\ge\aleph_{\omega_1}$, the least cardinality of a well-orderable family of nonempty sets without a choice function is $\aleph_1$.
\end{corollary}

For $\kappa\ge\aleph_{\omega_1}$, the same argument describes the limitation of Corollary~\ref{cor:refinements}(2). Choose $X_\gamma\subseteq A$ of ambient cardinality $\aleph_\gamma$ and let $Q_\gamma=\{X_\gamma,A\setminus X_\gamma\}$. No countable partition $R$ can refine $\pi_\gamma Q_\gamma$ for every $\gamma<\omega_1$, for any choice of permutations $\pi_\gamma$: the unions of blocks of $R$ would then realize every cardinality $\aleph_\gamma$. Thus countably many refinements can always be treated by one $R$ after conjugation, while some family of $\aleph_1$ refinements cannot.

\section{Dependent choice}\label{sec:dc}

Countable completeness of the normal filter is a sufficient condition for dependent choice in a permutation model: it gives closure under countable sequences of members of the model formed in the ambient universe. The following proposition shows that this condition fails for $\Vpp$.

\begin{proposition}\label{prop:sigma}
The model $\Vpp$ is not closed under countable sequences from $\U$, and $\mathcal F$ is not $\sigma$-complete.
\end{proposition}

\begin{proof}
Choose $A_0\subseteq A$ of ambient cardinality $\aleph_1$ and an injection $h:A_0\to{}^\omega2$. For $n<\omega$, put $X_n=\{a\in A_0:h(a)(n)=1\}$. Each $X_n$ belongs to $\Vpp$.

If $P$ supported $\seq{X_n:n<\omega}$, some block of $P$ would contain distinct $a,b\in A_0$. Their transposition would fix every $X_n$, implying $h(a)=h(b)$, a contradiction. The same argument shows that $\bigcap_{n<\omega}\symop(X_n)$ contains no basic subgroup $\fix(P)$, although every $\symop(X_n)$ belongs to $\mathcal F$.
\end{proof}

For comparison, $\Vfp$ is closed under countable sequences from $\U$, since countably many countable supports have countable union. That closure gives $\DC$ in $\Vfp$. To prove $\DC$ in $\Vpp$, we instead verify the following condition of Karagila and Schilhan.

\begin{definition}[{\cite[Definition~8]{KSDC}}]\label{def:shift}
A normal filter $\mathcal F$ on $G$ is \emph{$(\omega+1)$-shift complete} if for every sequence $\seq{H_n:n<\omega}$ in $\mathcal F$ there are $\pi_n\in G$ such that, writing $K_n=\pi_nH_n\pi_n^{-1}$,
\[
 \bigcap_{n<\omega}K_n\in\mathcal F
 \qquad\text{and}\qquad
 \pi_m\pi_n^{-1}\in K_n\quad(n<m<\omega).
\]
\end{definition}

This is equivalent to shift completeness as defined in \cite[Definition~2]{KSDC}, as noted immediately after \cite[Definition~8]{KSDC}. We use the shorter term below.

\begin{lemma}\label{lem:shift}
The filter $\mathcal F$ of the countable partition model is shift complete.
\end{lemma}

\begin{proof}
Let $\seq{H_n:n<\omega}$ be a sequence in $\mathcal F$. Choose countable partitions $S_n$ such that $S_{n+1}$ refines $S_n$ and $\fix(S_n)\subseteq H_n$. We will replace this sequence by partitions $T_n$ admitting a common countable refinement, and choose permutations that carry each earlier block to the same replacement.

Form the forest whose level-$n$ nodes are the blocks of $S_n$, with their levels attached. We write $|s|$ for the cardinality of the block at a node $s$. A block at level $n+1$ is a child of the block at level $n$ containing it. There are countably many nodes. For each node $t$, choose a child $d(t)$ of cardinality $|t|$ if one exists, and any child otherwise. Let $b_t$ be the branch through $t$ which subsequently follows the chosen children, and put
\[
 \mathcal B=\{b_t:t\text{ is a node}\}.
\]
This is a countable set of branches meeting every node. The cardinalities along a branch are nonincreasing, and hence eventually constant. Let $\mu_b$ be the eventual cardinality along $b\in\mathcal B$.

We claim that for every node $s$,
\begin{equation}\label{eq:branchsum}
 \sum_{\substack{b\in\mathcal B\\s\in b}}\mu_b=|s|.
\end{equation}
We argue by induction on $|s|$. Follow the chosen children from $s$. If their cardinalities remain $|s|$ forever, then $\mu_{b_s}=|s|$. For infinite $|s|$, this gives the lower bound in~\eqref{eq:branchsum}, while countability of $\mathcal B$ gives the upper bound $|s|\cdot\aleph_0=|s|$. For finite $|s|$, a child of full cardinality must be the only child, so there is exactly one branch through $s$, giving equality.

Otherwise, let $t$ be the first node on this path whose children all have cardinality less than $|s|$; then $|t|=|s|$. By induction, the sums for the children equal their cardinalities. Since these children partition $t$, summing over them gives $\sum_{b\ni t}\mu_b=|t|$. If $|s|$ is infinite, this is a lower bound for the sum through $s$, and the same countability argument gives the upper bound. If $|s|$ is finite, every step from $s$ to $t$ has a unique child, so every branch through $s$ passes through $t$. This proves~\eqref{eq:branchsum}.

For each root $p\in S_0$, partition $p$ into nonempty sets $F_b$, one for each $b\in\mathcal B$ through $p$, with $|F_b|=\mu_b$. Equation~\eqref{eq:branchsum} permits this. For every node $s$, set
\[
 s^*=\bigcup\{F_b:b\in\mathcal B,\ s\in b\}.
\]
Then $|s^*|=|s|$, roots satisfy $p^*=p$, and the sets $c^*$ for children $c$ of $s$ partition $s^*$. Let $T_n=\{s^*:s\in S_n\}$ and $T_\omega=\{F_b:b\in\mathcal B\}$. These are countable partitions, $T_0=S_0$, and $T_\omega$ refines every $T_n$.

For each $n$, take the union of bijections $s\to s^*$ over $s\in S_n$, obtaining a permutation $\pi_n$. For any $k\le n$ and $s\in S_k$, its level-$n$ descendants partition $s$, so $\pi_n[s]=s^*$. In particular, for $n<m$ the permutation $\pi_m\pi_n^{-1}$ fixes every block of $T_n$ setwise. Consequently,
\[
 \pi_m\pi_n^{-1}\in\fix(T_n)
 =\pi_n\fix(S_n)\pi_n^{-1}\subseteq K_n,
\]
and $\fix(T_\omega)\subseteq\bigcap_{n<\omega}K_n$. This proves shift completeness.
\end{proof}

\begin{theorem}\label{thm:dc}
For every uncountable $\kappa$, the model $\Vpp$ satisfies $\DC$.
\end{theorem}

\begin{proof}
Over an ambient model of $\DC$, shift completeness of the normal filter implies $\DC$ in the associated permutation model \cite[Sections~2--3]{KSDC}. Apply Lemma~\ref{lem:shift}.
\end{proof}

\begin{remark}
The filter is not $(\omega+2)$-shift complete, since that would imply $\sigma$-completeness by \cite[Lemma~9]{KSDC}, contrary to Proposition~\ref{prop:sigma}.
\end{remark}

Banerjee proved $\Vpp\models\neg\DC_{\omega_1}$ for every uncountable $\kappa$ \cite[Theorem~4.8(5)]{Banerjee}. Together with Theorem~\ref{thm:dc} and the monotonicity of the dependent choice principles, this gives the complete answer for cardinal indices.

\begin{corollary}\label{cor:dc}
For every uncountable $\kappa$ and every infinite well-ordered cardinal $\lambda$,
\[
 \Vpp\models\DC_\lambda\quad\Longleftrightarrow\quad\lambda=\omega.
\]
\end{corollary}

\section{Comparison and open questions}\label{sec:comparison}

Table~\ref{tab:comparison} compares the basic Fraenkel model $\Vf$ and Bruce's finite partition model $\Vp$, both over countably many atoms, with the countable-support Fraenkel model $\Vfp$ and the countable partition model $\Vpp$, both over $\aleph_1$ atoms.

\begin{table}[htbp]
\centering
\small
\begin{tabular}{@{}lcccc@{}}
\toprule
 & $\Vf$ & $\Vfp$ & $\Vp$ & $\Vpp$ \\
\midrule
all ambient subsets of $A$ belong & fails & fails & holds & holds \\
choice on $[A]^2$ & fails & fails & fails & fails \\
$\ACWO$ & fails & holds & fails & holds \\
$\DC$ & fails & holds & fails & holds \\
closed under countable sequences from $\U$ & fails & holds & fails & fails \\
an infinite amorphous set exists & holds & fails & open & fails \\
$(\aleph(A),\aleph^*(A))$ & $(\omega,\omega)$ & $(\omega_1,\omega_1)$ & $(\omega,\omega)$ & $(\omega_1,\omega_1)$ \\
\bottomrule
\end{tabular}
\caption{Choice principles in the four permutation models.}\label{tab:comparison}
\end{table}

In all four models, failure of choice on $[A]^2$ implies failure of $\OP$, $\BPI$, and the Kinna--Wagner selection principle. Since pure sets are well-orderable in these models, $\KW_\alpha$ also fails for every ordinal $\alpha$. Multiple choice fails by the argument of Proposition~\ref{prop:mckw}, applied to an infinite block of a partition support, or to $A\setminus E$ for a subset support $E$; for $\Vp$, see also \cite[Theorem~4.5(4)]{Banerjee}.

The $\Vpp$ column follows from the preceding sections. The standard properties of the Fraenkel models are recorded in \cite{JechAC,HR}, and the finite partition model is studied in \cite{Bruce,Banerjee}. For completeness, the same countable family witnesses failure of $\AC_\omega$ in $\Vf$ and $\Vp$: use $\seq{[A]^n:1\le n<\omega}$. If a choice function were supported by a finite set $E$, respectively a finite partition $P$, its value on $[A]^n$ would have to be contained in $E$, respectively in the union of the finite blocks of $P$. Both sets have a fixed finite size, giving a contradiction for large $n$.

The positive $\ACWO$ entry for $\Vfp$ follows by the argument of Theorem~\ref{thm:wochoice}. Given a countable support $E$ fixing the family pointwise, choose a countable $S\supseteq E$ with $S\setminus E$ infinite. Every countable support $D$, enlarged to contain $E$, can be moved into $S$ by a permutation fixing $E$ pointwise: an injection of $D\setminus E$ into $S\setminus E$ extends to a permutation because both complements have cardinality $\aleph_1$. Thus each member of the family has an $S$-supported element, and the selected choice function is supported by $S$.

Closure under countable sequences gives $\DC$ in $\Vfp$. The ambient enumeration of the atoms witnesses failure of this closure in $\Vf$ and $\Vp$, while Proposition~\ref{prop:sigma} treats $\Vpp$. The invariant calculations for the finite-support models follow from their having only finitely many fixed atoms and finite ranges for supported ordinal-valued functions.

For arbitrary uncountable $\kappa$, the results for $\Vpp$ are collected in the following corollary and Table~\ref{tab:parameter}.

\begin{corollary}\label{cor:summary}
For every uncountable $\kappa$, the model $\Vpp$ satisfies $\AC_\omega$, $\DC$, and $\aleph(A)=\aleph^*(A)=\omega_1$, and fails $\DC_{\omega_1}$. It satisfies $\ACWO$ exactly when $\kappa<\aleph_{\omega_1}$; otherwise an $\omega_1$-indexed family witnesses failure of choice.
\end{corollary}

\begin{table}[htbp]
\centering
\begin{tabular}{@{}lccc@{}}
\toprule
 & $\ACWO$ & $\DC$ & $\DC_{\omega_1}$ \\
\midrule
$\aleph_1\le\kappa<\aleph_{\omega_1}$ & holds & holds & fails \\
$\kappa\ge\aleph_{\omega_1}$ & fails & holds & fails \\
\bottomrule
\end{tabular}
\caption{Choice principles in $\Vpp$ as the ambient cardinality $\kappa$ varies.}\label{tab:parameter}
\end{table}

Bruce asked whether an infinite amorphous set exists in $\Vp$ \cite[Section~5]{Bruce}; this is the open entry in Table~\ref{tab:comparison}. For the countable partition model, Corollary~\ref{cor:union} answers Banerjee's union question below $\aleph_{\omega_1}$, leaving the following cases.

\begin{question}\label{q:union}
For $\kappa\ge\aleph_{\omega_1}$, is every union of a well-orderable family of well-orderable sets well-orderable in $\Vpp$?
\end{question}

One can also allow supports with more blocks. Fix an infinite cardinal $\theta<\kappa$ and use partitions of $A$ into at most $\theta$ blocks. Banerjee considers this construction when $\kappa=\theta^+$ \cite[Remark~4.9]{Banerjee}. The arguments of Lemma~\ref{lem:orbits} and Theorem~\ref{thm:invariants} give $\aleph(A)=\aleph^*(A)=\theta^+$. The relation requiring each new atom to differ from all earlier ones then witnesses failure of $\DC_{\theta^+}$: every sequence of length below $\theta^+$ misses an atom, since $A$ is not well-orderable, but a branch would inject $\theta^+$ into $A$.

The proof of Lemma~\ref{lem:shift} also gives ordinary $\DC$ in this model. The forest now has at most $\theta$ nodes, and at most $|s|$ of the selected branches pass through any infinite node $s$: at each later level its descendants are disjoint nonempty subsets of $s$, and there are only countably many levels. The finitely many ancestors of $s$ add only finitely many selected branches. This replaces the countability bound in~\eqref{eq:branchsum}; the rest of the proof is unchanged, with at most $\theta$ blocks in $T_\omega$.

For regular $\theta$ satisfying $\theta^{<\theta}=\theta$, the filter is $\theta$-complete, since fewer than $\theta$ partitions with at most $\theta$ blocks have a common refinement with at most $\theta$ blocks. This yields $\DC_{<\theta}$. A sufficient condition for $\DC_\theta$ would be $(\theta+1)$-shift completeness \cite[Section~5]{KSDC}. These observations motivate the following two conjectures, which extend the results for $\theta=\omega$.

\begin{question}\label{q:theta}
For an infinite regular cardinal $\theta<\kappa$, does the model with at most $\theta$ blocks in each support satisfy $\ACWO$ exactly when $\kappa<\aleph_{\theta^+}$? Does it satisfy $\DC_\lambda$ for every infinite cardinal $\lambda\le\theta$?
\end{question}

The model $\Vpp$ satisfies $\DC$ while failing multiple choice, every Kinna--Wagner principle, and the ordering principle. Below $\aleph_{\omega_1}$ it also satisfies $\ACWO$.

\begin{question}\label{q:other}
Which further weak choice principles in the catalogue \cite{HR} hold in $\Vpp$, and how do they depend on $\kappa$?
\end{question}

\section*{Acknowledgments}

This paper forms part of the author's Ph.D.\ thesis at the University of Toronto.  The author thanks her supervisors, Spencer Unger and Asaf Karagila, for their guidance and support, and thanks Asaf Karagila in particular for his detailed comments on earlier drafts of this paper.

\end{document}